\documentclass[11pt]{article}
\usepackage{amsmath}
\usepackage{amssymb}
\usepackage{amsthm}
\usepackage[usenames]{color}
\usepackage{amscd}
\usepackage{dsfont}
\usepackage{indentfirst}

\usepackage[colorlinks=true,linkcolor=blue,filecolor=red,
citecolor=webgreen]{hyperref}
\definecolor{webgreen}{rgb}{0,.5,0}

\numberwithin{equation}{section}

\def\C{{\mathds{C}}}

\def\N{{\mathds{N}}}
\def\Z{{\mathds{Z}}}

\def\1{{\bf 1}}

\def\id{\operatorname{id}}
\def\Pr{\operatorname{Pr}}

\newtheorem{theorem}{Theorem}[section]

\newtheorem{lemma}[theorem]{Lemma}
\newtheorem{cor}[theorem]{Corollary}

\newtheorem{remark}{Remark}[section]

\begin{document}

\title{{\bf Menon-type identities again: A note on a paper by Li, Kim and Qiao} }
\author{Pentti Haukkanen \\ Faculty of Information Technology and Communication Sciences (ITC)\\
Tampere University, FI-33014 Finland\\ E-mail: {\tt pentti.haukkanen@tuni.fi} \\ \ \\
L\'aszl\'o T\'oth \\ Department of Mathematics, University of P\'ecs \\
Ifj\'us\'ag \'utja 6, 7624 P\'ecs, Hungary \\ E-mail: {\tt ltoth@gamma.ttk.pte.hu}}
\date{}
\maketitle

\centerline{Publ. Math. Debrecen {\bf 96} (2020), 487--502}

\begin{abstract} We give common generalizations of the Menon-type identities by Sivaramakrishnan (1969) and  Li, Kim, Qiao (2019).
Our general identities involve arithmetic functions of several variables, and also contain, as special cases, identities for gcd-sum type
functions. We point out a new Menon-type identity concerning the lcm function. We present a simple character free approach for the proof.
\end{abstract}



\section{Introduction}

Menon's classical identity \cite{Men1965} states that for every $n\in \N:=\{1,2,\ldots \}$,
\begin{equation} \label{Menon_id}
M(n):= \sum_{\substack{a=1 \\ (a,n)=1}}^n (a-1,n) = \varphi(n) \tau(n),
\end{equation}
where $(a,n)$ stands for the greatest common divisor of $a$ and $n$, $\varphi(n)$ is Euler's totient function and
$\tau(n)=\sum_{d\mid n} 1$ is the divisor function.

Menon \cite{Men1965} proved this identity by three distinct methods, the first one being based on the Cauchy-Frobenius-Burnside lemma on group
actions. This method was used later by Sury \cite{Sur2009}, T\'oth \cite{Tot2011}, Li and Kim \cite{LiKim2017}, and other authors to
derive different generalizations and analogs of \eqref{Menon_id}. Number theoretic methods were also applied in several papers to deduce various
Menon-type identities. See, e.g., \cite{Hau2005,HW1996,HauWan1997,LiHuKimTaiw,LiKim2018,LKQ2019,Tot2018,Tot2019,Tot,WZJ2019,ZhaCao2017}.

It is less known and is not considered in the above mentioned papers the following old generalization, due to
Sivaramakrishnan \cite{Siv1969}, in a slightly different form:
\begin{equation} \label{Menon_id_Siv}
M(m,n,t):= \sum_{\substack{a=1 \\ (a,m)=1}}^t (a-1,n) = \frac{t\, \varphi(m)\tau(n)}{m} \prod_{p^\nu\mid \mid n_1}
\left(1-\frac{\nu}{(\nu+1)p} \right),
\end{equation}
where $m,n,t\in \N$ such that $m\mid t$, $n\mid t$ and $n_1=\max \{d\in \N: d\mid n, (d,m)=1\}$. If $m=n=t$, then
$M(n,n,n)=M(n)$, that is, \eqref{Menon_id_Siv} reduces to \eqref{Menon_id}. However, if $n\mid m$ and $t=m$, then it
follows from \eqref{Menon_id_Siv} that
\begin{equation*}
\sum_{\substack{a=1 \\ (a,m)=1}}^m (a-1,n) = \varphi(m) \tau(n),
\end{equation*}
which was recently obtained by Jafari and Madadi \cite[Cor.\ 2.2]{JafMad2017}, using group theoretic arguments,
without referring to the paper \cite{Siv1969}. It was pointed out by Sivaramakrishnan \cite{Siv1969} that if $t=[m,n]$, the least common
multiple of $m$ and $n$, then $M(m,n,[m,n])$ is a multiplicative function of two variables.

In a quite recent paper, Li, Kim and Qiao \cite[Th.\ 2.5]{LKQ2019} proved that for any integers
$n\ge 1$, $k\ge 0$, $\ell \ge 1$ one has
\begin{equation} \label{id_Debrecen_paper}
\sum_{\substack{1\le a,b_1,\ldots,b_k \le n \\ (a,n)=1}} (a^{\ell}-1,b_1,\ldots,b_k,n)=\varphi(n) (\id_k \ast \, C^{(\ell)})(n),
\end{equation}
where $\ast$ denotes the Dirichlet convolution, $\id_k(n)=n^k$ and $C^{(\ell)}(n)$ is the number of solutions of the congruence $x^{\ell}
\equiv 1$ (mod $n$) with $(x, n)=1$. Note that the condition $(x, n)=1$ can be omitted here.
For the proof they used properties of characters of finite abelian groups.
The case $k=0$ recovers certain identities given by the second author \cite{Tot2011}. If $k=0$ and $\ell=1$, then \eqref{id_Debrecen_paper}
reduces to \eqref{Menon_id}.

The sum $M(n)$ is related to the gcd-sum function, also known as Pillai's arithmetical function, given by
\begin{equation} \label{gcd_sum}
G(n):=\sum_{a=1}^n (a,n)= n \sum_{d\mid n} \frac{\varphi(d)}{d} \quad (n\in \N).
\end{equation}

A large number of different generalizations and analogs of the function $G(n)$ is presented in the literature. See, e.g.,
\cite{Hau2008,Tot2010}.

It is the goal of this paper to give common generalizations of the identities \eqref{Menon_id_Siv} and \eqref{id_Debrecen_paper},
and to present a simple character free approach for the proof. Our general identity, included in Theorem \ref{Theorem_main}, involves
arithmetic functions of several variables, and also contains, as a special case, identities for gcd-sum type functions, such as
identity \eqref{gcd_sum}. The identity of Theorem \ref{Theorem_g} is concerning arithmetic functions of a single variable.

We point out the following new Menon-type identity, which is another special case of our results. See Theorem \ref{Th_Menon_lcm} and Corollary
\ref{Cor_lcm_n}. If $n\in \N$, then
\begin{equation} \label{a_b_lcm}
\sum_{\substack{1\le a,b\le n\\ (a,n)=(b,n)=1}} [(a-1,n),(b-1,n)] =\varphi(n)^2 \prod_{p^{\nu}\mid\mid n}
\left(1+ 2\nu -\frac{p^{\nu}-1}{p^{\nu-1}(p-1)^2} \right).
\end{equation}

Note that identity \eqref{Menon_id_Siv} was generalized by Sita Ramaiah \cite[Th.\ 9.1]{Sit1978} in another
way, namely in terms of regular convolutions. Our results can further be generalized to regular convolutions and to $k$-reduced residue systems.
For the sake of brevity, we do not present the details. For appropriate material we refer to \cite{Coh1949,Coh1956,Sit1978}.

\section{Preliminaries}

\subsection{Arithmetic functions of several variables} \label{Sect_Prel_Arithm_Func}

Let $f, g:\N^k\to\C$ be arithmetic functions of $k$ variables. Their Dirichlet convolution is defined as
\begin{equation*}
(f\ast_k g)(n_1, \ldots, n_k) =\sum_{d_1\mid n_1,\ldots,d_k\mid n_k} f(d_1, \ldots, d_k)g(n_1/d_1,\ldots, n_k/d_k).
\end{equation*}

In the case $k=1$ we write simply $f\ast_1 g=f\ast g$.  The identity under $\ast_k$ is
\begin{equation*}
\delta_k(n_1, \ldots, n_k)=\delta(n_1)\cdots\delta(n_k),
\end{equation*}
where $\delta(1)=1$ and $\delta(n)=0$ for $n\ne 1$. An arithmetic function $f$ of $k$ variables possesses an inverse under $\ast_k$
if and only if $f(1,\ldots, 1)\ne 0$.  Let $\zeta_k(n_1, \ldots, n_k)$ be defined as
$\zeta_k(n_1, \ldots, n_k)=1$ for all $n_1, \ldots, n_k\in\N$. Its Dirichlet inverse is the M\"{o}bius function $\mu_k$ of $k$
variables given as
\begin{equation*}
\mu_k(n_1, \ldots, n_k) = \mu(n_1)\cdots\mu(n_k),
\end{equation*}
where $\mu$ is the classical M\"{o}bius function of one variable.

Let $g$ be an arithmetic function of one variable. Then the principal function $\Pr_k(g)$ associated with $g$ is the arithmetic
function of $k$ variables defined as
\begin{equation*}
\Pr_k(g)(n_1,\ldots, n_k)
= \begin{cases}
g(n),  & \text{ if $n_1=\cdots = n_k=n$,}\\
0,  & \text{ otherwise.}
\end{cases}
\end{equation*}
(See Vaidyanathaswamy \cite{Vai1931}.) Let $f$ be the arithmetic function of $k$ variables defined by
\begin{equation*}
f(n_1,\ldots, n_k)=g((n_1,\ldots, n_k)),
\end{equation*}
having the gcd on the right-hand side. Then
\begin{equation*}
f(n_1,\ldots, n_k)
=\sum_{d\mid (n_1,\ldots, n_k)}(\mu\ast g)(d) =\sum_{d_1\mid n_1,\ldots, d_k\mid n_k}
\Pr_k(\mu\ast g)(d_1,\ldots, d_k),
\end{equation*}
that is,
\begin{equation*}
f= \Pr_k(\mu\ast g) \ast_k \zeta_k,
\end{equation*}
which means that
\begin{equation}\label{eq:princ}
\Pr_k(\mu\ast g)= \mu_k\ast_k f.
\end{equation}

An arithmetic function $f$ of $k$ variables is said to be multiplicative if
$f(1,\ldots, 1)=1$  and
\begin{equation*}
f(m_1 n_1,\ldots, m_k n_k)=f(m_1,\ldots, m_k)f(n_1,\ldots, n_k)
\end{equation*}
for all positive integers $m_1,\ldots, m_k$ and $n_1,\ldots, n_k$
with $(m_1\cdots m_k, n_1\cdots n_k)=1$. For example, the gcd function $(n_1,\ldots,n_k)$ and the lcm function $[n_1,\ldots,n_k]$ are
multiplicative. If $f$ and $g$ are multiplicative functions of $k$ variables, then their Dirichlet convolution
$f\ast_k g$ is also multiplicative. See \cite{Tot2014,Vai1931}.

\subsection{Number of solutions of congruences} \label{Sect_Nr_Congr}

For a given polynomial $P\in \Z[x]$ let $N_P(n)$ denote the number of solutions $x$ (mod $n$) of the congruence
$P(x)\equiv 0$ (mod $n$) and let $\widehat{N}_P(n)$ be the number of solutions $x$ (mod $n$) such that $(x,n)=1$.
Furthermore, for a fixed integer $s\in \N$, let $\widehat{N}_P(n,s)$ be the number of solutions $x$ (mod $n$) such that
$(x,n,s)=1$.

The functions $N_P(n)$, $\widehat{N}_P(n)$ and $\widehat{N}_P(n,s)$ are multiplicative in $n$, which are direct
consequences of the Chinese remainder theorem.

It is easy to see that if $P(x)=a_0+a_1x+\cdots + a_kx^k$ and $(a_0,n)=1$, then $N_P(n)=\widehat{N}_P(n)=\widehat{N}_P(n,s)$.
This applies, in particular, to $P(x)=-1+x^{\ell}$. See the Introduction regarding the notation $C^{(\ell)}(n)$,
used by Li, Kim and Qiao \cite{LKQ2019}.

\subsection{Lemma}

We will need the next lemma.

\begin{lemma} \label{Lemma_cong} Let $d,r,s\in \N$, $x\in \Z$ such that $d\mid r$, $s\mid r$. Then
\begin{equation*}
\sum_{\substack{1\le a\le r \\ (a,s)=1 \\ a\equiv x \, \text{\rm (mod $d$)} }} 1 =
\begin{cases}  \displaystyle{ \frac{r}{d} \prod_{\substack{p\mid s\\ p\nmid d}} \left(1-\frac1{p}\right)}, & \text{ if $(d,s,x)=1$,} \\
0, & \text{ otherwise.}
\end{cases}
\end{equation*}
\end{lemma}

In the special case $r=s$ this is known in the literature, usually proved  by the
inclusion-exclusion principle. See, e.g., \cite[Th.\ 5.32]{Apo1976}. Also see \cite[Lemma]{HauWan1997} for a generalization in terms of regular
convolutions. Here we use a different approach, similar to the proof of \cite[Th.\ 9.1]{Sit1978} and to the proofs of our previous papers \cite{Tot2019,Tot}.

\begin{proof}[Proof of Lemma {\rm \ref{Lemma_cong}}] Let $A$ denote the given sum.
If $(d,s,x)\ne 1$, then the sum $A$ is empty and equal to zero. Indeed, if we assume $p\mid (d,s,x)$ for some prime $p$, then $p\mid a\equiv x$
(mod $d$). Hence $p\mid (a,s)=1$, a contradiction.

Assume now that $(d,s,x)=1$. By using the property of the M\"{o}bius function, the given sum can be written as
\begin{equation} \label{last_sum}
A = \sum_{\substack{1\le a \le r \\ a\equiv x \, \text{\rm (mod $d$)} }}  \sum_{\delta \mid (a,s)}
\mu(\delta) = \sum_{\delta \mid s} \mu(\delta) \sum_{\substack{j=1\\ \delta j\equiv x \, \text{\rm (mod $d$)}  }}^{r/\delta} 1.
\end{equation}

Let $\delta\mid s$ be fixed. The linear congruence $\delta j\equiv x$ (mod $d$) has solutions in $j$ if and only if $(\delta,d) \mid x$. Here
$(\delta,d)\mid \delta$ and $\delta \mid s$, hence $(\delta,d) \mid s$. Also, $(\delta,d)\mid d$. If $(\delta,d)\mid x$ holds, then $(\delta,d)\mid (d,s,x)=1$, therefore
$(\delta,d)=1$. We deduce that the above congruence has
\begin{equation*}
N= \frac{r}{d \delta}
\end{equation*}
solutions (mod $r/\delta$) and the last sum in \eqref{last_sum} is $N$.
This gives
\begin{equation*}
A= \frac{r}{d} \sum_{\substack{\delta \mid s\\ (\delta,d)=1}} \frac{\mu(\delta)}{\delta}= \frac{r}{d}\prod_{\substack{p\mid s\\ p\nmid d}}
\left(1-\frac1{p}\right).
\end{equation*}
\end{proof}

\section{Main results}

Assume that

(1) $k,\ell \ge 0$ are fixed integers, not both zero;

(2) $m_i,r_i,s_i,n_j,t_j\in \N$ are integers such that $m_i\mid r_i$, $s_i\mid r_i$,  $n_j\mid t_j$ ($1\le i\le k$, $1\le j\le \ell$);

(3) $f:\N^{k+\ell} \to \C$ is an arbitrary arithmetic function of $k+\ell$ variables;

(4) $P_i,Q_j\in \Z[x]$ are arbitrary polynomials ($1\le i\le k$, $1\le j\le \ell$).

Consider the sum
\begin{equation*}
S:= \sum_{\substack{1\le a_i\le r_i\\ (a_i,s_i)=1\\ 1\le i\le k}}
\sum_{\substack{1\le b_j\le t_j\\ 1\le j\le \ell}} f((P_1(a_1),m_1), \ldots,(P_k(a_k),m_k), (Q_1(b_1),n_1),\ldots,(Q_{\ell}(b_{\ell}),n_{\ell})),
\end{equation*}
where $(P_i(a_i),m_i)$ and $(Q_j(b_j),n_j)$ represent the gcd's of the corresponding values ($1\le i\le k$, $1\le j\le \ell$).

\begin{theorem} \label{Theorem_main}
Under the above assumptions (1)-(4) we have
\begin{equation*}
S =  r_1\cdots r_k t_1\cdots t_{\ell} \sum_{\substack{d_i\mid m_i \\1\le i\le k}}
\sum_{\substack{e_j \mid n_j \\ 1\le j\le \ell}} \frac{(\mu_{k+\ell}\ast_{k+\ell}f)(d_1,\ldots,d_k,e_1,\ldots,e_{\ell})}{d_1\cdots d_k
e_1\cdots e_{\ell}}
\end{equation*}
\begin{equation*}
\times \left(\prod_{1\le i\le k} \widehat{N}_{P_i}(d_i,s_i) \beta(s_i,d_i)\right)\left( \prod_{1\le j  \le \ell} N_{Q_j}(e_j)\right),
\end{equation*}
where $\widehat{N}_{P_i}(d_i,s_i)$ and $N_{Q_j}(e_j)$ \textup{($1\le i\le k$, $1\le j\le \ell$)} are defined in Section \ref{Sect_Nr_Congr}, and
\begin{equation*}
\beta(s_i,d_i)= \prod_{\substack{p\mid s_i\\ p\nmid d_i}} \left(1-\frac1{p}\right).
\end{equation*}
\end{theorem}

\begin{cor}
If $\ell=0$, then Theorem \ref{Theorem_main} gives the pure Menon-type identity
\begin{equation}\label{Cor_pure_Menon}
S =  r_1\cdots r_k
\sum_{\substack{d_i\mid m_i \\1\le i\le k}}
 \frac{(\mu_{k}\ast_{k}f)(d_1,\ldots,d_k)}{d_1\cdots d_k}
\left(\prod_{1\le i\le k}  \widehat{N}_{P_i}(d_i,s_i) \beta(s_i,d_i)\right),
\end{equation}
and if $k=0$, it gives the pure gcd-sum identity
\begin{equation*}
S = t_1\cdots t_{\ell}
\sum_{\substack{e_j \mid n_j \\ 1\le j\le \ell}} \frac{(\mu_{\ell}\ast_{\ell}f)(e_1,\ldots,e_{\ell})}{e_1\cdots e_{\ell}}
\left(\prod_{1\le j  \le \ell} N_{Q_j}(e_j)\right).
\end{equation*}
\end{cor}

If $k=1$, $f(n)=n$ ($n\in \N$) and $P(x)=x-1$, then identity \eqref{Cor_pure_Menon} reduces to
\begin{equation*}
\sum_{\substack{a=1 \\ (a,s)=1}}^r (a-1,m)  = r \sum_{d\mid m} \frac{\varphi(d)}{d} \prod_{\substack{p\mid s\\ p\nmid d}}
\left(1-\frac1{p}\right)
\end{equation*}
\begin{equation*}
= \frac{r\, \varphi(s)\tau(m)}{s} \prod_{p^\nu\mid \mid m_1}
\left(1-\frac{\nu}{(\nu+1)p} \right),
\end{equation*}
where $m\mid r$, $s\mid r$ and $m_1=\max \{d\in \N: d\mid m, (d,s)=1\}$, which is identity \eqref{Menon_id_Siv}
(with the corresponding change of notations).

\begin{remark}{\rm Haukkanen and Wang \cite{HauWan1997} considered
systems of polynomials in several variables and a different constraint,
namely $(a_1,\ldots,a_k,n)=1$ in the first sum defining $S$. }
\end{remark}

\begin{proof}[Proof of Theorem {\rm \ref{Theorem_main}}]
It is an immediate consequence of the definition of the function $\mu_k$ that
\begin{equation} \label{key_id}
f(n_1,\ldots,n_k) = \sum_{d_1\mid n_1,\ldots, d_k\mid n_k} (\mu_k\ast_k f)(d_1,\ldots,d_k).
\end{equation}

By using \eqref{key_id} we have
\begin{equation*}
S = \sum_{\substack{1\le a_i\le r_i\\ (a_i,s_i)=1\\ 1\le i\le k}}
\sum_{\substack{1\le b_j\le t_j\\ 1\le j\le \ell}} \sum_{\substack{d_i\mid (P_i(a_i),m_i)\\1\le i\le k}}
\, \sum_{\substack{e_j \mid (Q_j(b_j),n_j)\\ 1\le j\le \ell}} (\mu_{k+\ell}\ast_{k+\ell} f)(d_1,\ldots,d_k,e_1,\ldots,e_{\ell})
\end{equation*}
\begin{equation*}
= \sum_{\substack{d_i\mid m_i \\1\le i\le k}} \sum_{\substack{e_j \mid n_j \\ 1\le j\le \ell}} (\mu_{k+\ell}\ast_{k+\ell}f)(d_1,\ldots,d_k,e_1,\ldots,e_{\ell})
\end{equation*}
\begin{equation*}
\times \Big( \prod_{1\le i\le k}  \sum_{\substack{1\le a_i\le r_i\\ (a_i,s_i)=1\\ P_i(a_i)\equiv 0 \text{ (mod $d_i$)} }} 1 \Big)
\Big( \prod_{1\le j \le \ell} \sum_{\substack{1\le b_j\le t_j\\ Q_j(b_j)\equiv 0 \text{ (mod $e_j$) } }} 1 \Big).
\end{equation*}

Now we use Lemma \ref{Lemma_cong} to evaluate the sum
\begin{equation*}
B_i:= \sum_{\substack{1\le a_i\le r_i\\ (a_i,s_i)=1\\ P_i(a_i)\equiv 0 \text{ (mod $d_i$)} }} 1.
\end{equation*}

For any $x$ such that $(x,d_i,s_i)=1$ we have
\begin{equation*}
\sum_{\substack{1\le a_i\le r_i\\ (a_i,s_i)=1\\ a_i\equiv x \text{ (mod $d_i$)} }} 1 = \frac{r_i}{d_i} \beta(s_i,d_i),
\end{equation*}
and there are $\widehat{N}_{P_i}(d_i,s_i)$ such values of $x$ (mod $d_i$). Hence,
\begin{equation*}
B_i= \frac{r_i}{d_i} \beta(s_i,d_i) \widehat{N}_{P_i}(d_i,s_i).
\end{equation*}

We also have
\begin{equation*}
\sum_{\substack{1\le b_j\le t_j\\ Q_j(b_j)\equiv 0 \text{ (mod $e_j$)} }} 1 = \frac{t_j}{e_j} N_{Q_j}(e_j).
\end{equation*}

Notice that here $r_i/d_i$ and $t_j/e_j$ are integers for any $i,j$.

Putting together this gives
\begin{equation*}
S = \sum_{\substack{d_i\mid m_i \\1\le i\le k}} \sum_{\substack{e_j \mid n_j \\ 1\le j\le \ell}} (\mu_{k+\ell}\ast_{k+\ell}f)(d_1,\ldots,d_k,e_1,\ldots,e_{\ell})
\end{equation*}
\begin{equation*}
\times \Big( \prod_{1\le i\le k}  \widehat{N}_{P_i}(d_i,s_i) \frac{r_i}{d_i} \beta(s_i,d_i) \Big)
\Big(\prod_{1\le j \le \ell} \frac{t_j}{e_j} N_{Q_j}(e_j)\Big)
\end{equation*}
\begin{equation*}
= r_1\cdots r_k t_1\cdots t_{\ell} \sum_{\substack{d_i\mid m_i \\1\le i\le k}}
\sum_{\substack{e_j \mid n_j \\ 1\le j\le \ell}} \frac{(\mu_{k+\ell}\ast_{k+\ell}f)(d_1,\ldots,d_k,e_1,\ldots,e_{\ell})}{d_1\cdots d_k
e_1 \cdots e_{\ell}}
\end{equation*}
\begin{equation*}
\times \Big(\prod_{1\le i\le k}  \widehat{N}_{P_i}(d_i,s_i) \beta(s_i,d_i) \Big) \Big(\prod_{1\le j  \le \ell} N_{Q_j}(e_j) \Big).
\end{equation*}
\end{proof}

\begin{cor} Assume that $m_i\mid s_i$ and $s_i\mid r_i$ for any $i$ with $1\le i\le k$ Then
\begin{equation*}
S =  r_1\frac{\varphi(s_1)}{s_1}\cdots r_k\frac{\varphi(s_k)}{s_k} t_1\cdots t_{\ell} \sum_{\substack{d_i\mid m_i \\1\le i\le k}}
\sum_{\substack{e_j \mid n_j \\ 1\le j\le \ell}} \frac{(\mu_{k+\ell}\ast_{k+\ell}f)(d_1,\ldots,d_k,e_1,\ldots,e_{\ell})}{\varphi(d_1)
\cdots \varphi(d_k) e_1\cdots e_{\ell}}
\end{equation*}
\begin{equation*}
\times \Big(\prod_{1\le i\le k} \widehat{N}_{P_i}(d_i) \Big) \Big( \prod_{1\le j  \le \ell} N_{Q_j}(e_j) \Big).
\end{equation*}
\end{cor}

\begin{proof} Apply Theorem \ref{Theorem_main}. Since $d_i\mid m_i$, we have $d_i\mid s_i$. Hence
$\widehat{N}_{P_i}(d_i,s_i)= \widehat{N}_{P_i}(d_i)$ and
\begin{equation*}
\beta(s_i,d_i)= \prod_{\substack{p\mid s_i\\ p\nmid d_i}} \left(1-\frac1{p}\right) = \frac{\varphi(s_i)/s_i}{\varphi(d_i)/d_i}.
\end{equation*}
\end{proof}

\begin{cor} Assume that $m_i=r_i=s_i$ and $n_j=t_j$ for any $i,j$ ($1\le i\le k$, $1\le j\le \ell$). Then
\begin{equation} \label{S_special}
S =  \varphi(m_1)\cdots \varphi(m_k) n_1\cdots n_{\ell}
\sum_{\substack{d_i\mid m_i \\1\le i\le k}} \sum_{\substack{e_j \mid n_j \\ 1\le j\le \ell}}
\frac{(\mu_{k+\ell}\ast_{k+\ell}f)(d_1,\ldots,d_k,e_1,\ldots,e_{\ell})}{\varphi(d_1)\cdots \varphi(d_k) e_1\cdots e_{\ell}}
\end{equation}
\begin{equation*}
\times \Big( \prod_{1\le i\le k}  \widehat{N}_{P_i}(d_i) \Big)  \Big( \prod_{1\le j  \le \ell} N_{Q_j}(e_j) \Big).
\end{equation*}
\end{cor}

\begin{theorem} Assume conditions (1)-(4). Furthermore, let $r_i=[m_i,s_i]$, $t_j=n_j$ \textup{($1\le i\le k$, $1\le j\le \ell$)}
and let $f$ be a multiplicative function of $k+\ell$ variables.  Then the sum
$$
S=S(m_1,\ldots,m_k,s_1,\ldots,s_k,n_1,\ldots,n_{\ell})
$$
represents a multiplicative function of $2k+\ell$ variables.
\end{theorem}

\begin{proof} Note that
\begin{equation} \label{g}
\beta(s_i,d_i)= \sum_{\substack{\delta \mid s_i\\ (\delta,d_i)=1}} \frac{\mu(\delta)}{\delta}=
\sum_{\delta \mid s_i} \frac{\mu(\delta)}{\delta} h(\delta,d_i),
\end{equation}
where the function of two variables
\begin{equation*}
h(\delta,d_i)= \sum_{\substack{c\mid \delta \\ c\mid d_i}} \mu(c)
\end{equation*}
is multiplicative, being the convolution of multiplicative functions. Therefore, $\beta(s_i,d_i)$, given by the
convolution \eqref{g} is also multiplicative.

We conclude that $S$, given in Theorem \ref{Theorem_main} as a convolution of $2k+\ell$ variables of multiplicative functions,
is multiplicative, as well.
\end{proof}

\begin{cor} \label{Cor_multipl_k_ell} Assume that $m_i=r_i=s_i$ and $n_j=t_j$ for any $i,j$ and $f$ is multiplicative, viewed as a function of $k+\ell$ variables. Then
$S$ given by \eqref{S_special} is also multiplicative in $m_1,\ldots,m_k,n_1,\ldots,n_{\ell}$, as a function of $k+\ell$ variables.
\end{cor}

\begin{remark} {\rm Note that in his original paper Menon \cite[Lemma]{Men1965} proved that if $f$ is a multiplicative arithmetic function of $r$
variables and $P_i\in \Z[x]$ are polynomials, then the function
\begin{align*}
F(n):= \sum_{a=1}^n f((P_1(a),n),\ldots,(P_r(a),n))
\end{align*}
is multiplicative in the single variable $n$. Here $F(n)$ is not a special case our sum $S$, but it can be treated in a similar way. By using
\eqref{key_id} one obtains the formula
\begin{equation} \label{F_convo}
F(n)= n \sum_{d_1\mid n,\ldots, d_r\mid n} \frac{(\mu_r \ast_r f)(d_1,\ldots,d_r)}{[d_1,\ldots,d_r]} N(d_1,\ldots,d_r),
\end{equation}
valid for any function $f$ of $r$ variables, where $N(d_1,\ldots,d_r)$ is the number of solutions (mod $[d_1,\ldots,d_r]$) of the simultaneous
congruences $P_1(x)\equiv 0$ (mod $d_1$), ..., $P_r(x)\equiv 0$ (mod $d_r$). Note that $N(d_1,\ldots,d_r)$ is a multiplicative function of $r$ variables.
If $f$ is multiplicative, then the convolution representation \eqref{F_convo} shows that $F$ is also multiplicative.
}
\end{remark}

In what follows assume that

(1') $k,\ell \ge 0$ are fixed integers, not both zero;

(2') $n,r_i,s_i,t_j\in \N$ are integers such that $n\mid r_i$, $s_i\mid r_i$, $n\mid t_j$ ($1\le i\le k$, $1\le j\le \ell$);

(3') $g:\N \to \C$ is an arbitrary arithmetic function;

(4') $P_i,Q_j\in \Z[x]$ are arbitrary polynomials ($1\le i\le k$, $1\le j\le \ell$).

Consider the sum
\begin{equation*}
T:= \sum_{\substack{1\le a_i\le r_i\\ (a_i,s_i)=1\\ 1\le i\le k}} \sum_{\substack{1\le b_j\le t_j\\ 1\le j\le \ell}}
g((P_1(a_1), \ldots,P_k(a_k), Q(b_1),\ldots,Q_{\ell}(b_{\ell}),n)),
\end{equation*}
with the gcd on the right hand side.

We have the following result.

\begin{theorem} \label{Theorem_g}  Assume conditions (1')-(4'). Then
\begin{equation*}
T =  r_1\cdots r_k t_1\cdots t_{\ell} \sum_{d\mid n} \frac{(\mu\ast g)(d)}{d^{k+\ell}}
\left(\prod_{1\le i\le k}  \widehat{N}_{P_i}(d,s_i) \beta(s_i,d) \right) \left(\prod_{1\le j  \le \ell} N_{Q_j}(d)\right),
\end{equation*}
where
\begin{equation*}
\beta(s_i,d)= \prod_{\substack{p\mid s_i\\ p\nmid d}} \left(1-\frac1{p}\right)
\end{equation*}
\end{theorem}

\begin{proof} Apply Theorem \ref{Theorem_main} in the case when $m_i=n_j=n$ ($1\le i\le k$, $1\le j\le \ell$) and
\begin{equation*}
f(x_1,\ldots,x_k,y_1,\ldots,y_{\ell}) = g((x_1,\ldots,x_k,y_1,\ldots,y_{\ell})).
\end{equation*}

Then
\begin{equation*}
f((P_1(a_1),m_1), \ldots,(P_k(a_k),m_k), (Q(b_1),n_1),\ldots,(Q_{\ell}(b_{\ell}),n_{\ell}))
\end{equation*}
\begin{equation*}
=  g((P_1(a_1), \ldots,P_k(a_k),Q(b_1),\ldots,Q_{\ell}(b_{\ell}),n)).
\end{equation*}

From (\ref{eq:princ}) we obtain
\begin{equation*}
(\mu_{k+\ell}\ast_{k+\ell}f)(x_1,\ldots,x_k,y_1,\ldots,y_{\ell})= \begin{cases} (\mu\ast g)(n), & \text{ if $x_1=\cdots = x_k=y_1=\ldots=y_{\ell}=
n$,}\\ 0, & \text{ otherwise.}
\end{cases}
\end{equation*}
\end{proof}

In the special case $g(n)=n$, $Q_j(x)=x$, $r_i=s_i=t_j=n$ ($1\le i\le k$, $1\le j\le \ell$) we obtain from Theorem \ref{Theorem_g} the next result.

\begin{cor}
\begin{equation*}
\sum_{\substack{1\le a_i\le n\\ (a_i,n)=1\\ 1\le i\le k}}
\sum_{\substack{1\le b_j\le n\\ 1\le j\le \ell}} (P_1(a_1), \ldots,P_k(a_k),b_1,\ldots,b_{\ell},n)
 =  \varphi(n)^{k} (\id_\ell\ast G_k)(n),
\end{equation*}
where
\begin{equation*}
G_k(n)=\varphi(n)^{1-k}\prod_{1\le i\le k}  \widehat{N}_{P_i}(n).
\end{equation*}
\end{cor}

If $P_i(x)=x^{q_i}-1$ ($1\le i\le k$), then we obtain

\begin{cor} If $q_i\in \N$ \textup{($1\le i \le k$)}, then
\begin{equation*}
\sum_{\substack{1\le a_i\le n\\ (a_i,n)=1\\ 1\le i\le k}}
\sum_{\substack{1\le b_j\le n\\ 1\le j\le \ell}} (a_1^{q_1}-1,\ldots,a_k^{q_k}-1,b_1,\ldots,b_{\ell},n)
 =  \varphi(n)^{k} (\id_\ell\ast H_k)(n),
\end{equation*}
 where
\begin{equation} \label{G_k_n}
H_k(n)=\varphi(n)^{1-k}\prod_{1\le i\le k} C^{(q_i)}(n),
\end{equation}
$C^{(q_i)}(n)$ being the number of solutions of the congruence $x^{q_i} \equiv 1$ \textup{(mod $n$)}. \end{cor}

For $k=1$, \eqref{G_k_n} reduces to identity \eqref{id_Debrecen_paper} by Li, Kim and Qiao \cite{LKQ2019}.

Several other special cases can be discussed. For example, let $\ell=0$. By formula
\eqref{S_special} we have
\begin{equation} \label{V}
V(n_1,\ldots,n_k): = \sum_{\substack{1\le a_i\le n_i\\ (a_i,n_i)=1\\ 1\le i\le k}} f((P_1(a_1),n_1), \ldots,(P_k(a_k),n_k))
\end{equation}
\begin{equation*}
=  \varphi(n_1)\cdots \varphi(n_k) \sum_{\substack{d_i\mid n_i \\1\le i\le k}}
\frac{(\mu_k\ast_k f)(d_1,\ldots,d_k)}{\varphi(d_1)\cdots \varphi(d_k)}
\Big( \prod_{1\le i\le k}  \widehat{N}_{P_i}(d_i) \Big).
\end{equation*}

If $f:\N \to \C$ is multiplicative, then $V(n_1,\ldots,n_k)$ is multiplicative, as well, by Corollary \ref{Cor_multipl_k_ell}.
For prime powers $p^{\nu_1},\ldots,p^{\nu_k}$ the values  $V(p^{\nu_1},\ldots,p^{\nu_k})$ can be computed in the case of special
functions $f$ and special polynomials $P_i$.

We confine ourselves with the case of the lcm function $f(n_1,\ldots,n_k)=[n_1,\ldots,n_k]$ and the polynomials $P_i(x)=x-1$
($1\le i\le k$), included in the next section.

\section{A special case}

In this section we consider the function
\begin{equation*}
W(n_1,\ldots,n_k): = \sum_{\substack{1\le a_i\le n_i\\ (a_i,n_i)=1 \\ 1\le i\le k}} [(a_1-1,n_1), \ldots,(a_k-1,n_k)].
\end{equation*}

\begin{theorem} \label{Th_Menon_lcm} For any $n_1,\ldots,n_k\in \N$,
\begin{equation*}
W(n_1,\ldots,n_k) =  \varphi(n_1)\cdots \varphi(n_k) h(n_1,\ldots,n_k),
\end{equation*}
where the function $h$ is multiplicative, symmetric in the variables and  for any prime powers $p^{\nu_1},\ldots, p^{\nu_k}$ such that
$\nu_1\ge \cdots \ge \nu_t\ge 1$, $\nu_{t+1}=\cdots = \nu_k=0$,
\begin{equation*}
h(p^{\nu_1},\ldots,p^{\nu_k})
\end{equation*}
\begin{equation*}
=  1+(\nu_1+\cdots +\nu_t)+ \sum_{j=1}^{t-1} \frac{(-1)^jp^j}{(p-1)^j(p^j-1)} \Big( \binom{t}{j+1} -
\sum_{\substack{M\subseteq \{1,\ldots,t\}\\ \#M=j+1}} \frac1{p^{j\nu_{\max M}}} \Big).
\end{equation*}
\end{theorem}

\begin{proof}
According to \eqref{V} we have
\begin{equation*}
W(n_1,\ldots,n_k) =  \varphi(n_1)\cdots \varphi(n_k) \sum_{\substack{d_i\mid n_i \\1\le i\le k}}
\frac{(\mu_k\ast_k f)(d_1,\ldots,d_k)}{\varphi(d_1)\cdots \varphi(d_k)},
\end{equation*}
where $f(n_1,\ldots,n_k)=[n_1,\ldots,n_k]$.

Here $W(n_1,\ldots,n_k)$ is multiplicative and we compute the values $W(p^{\nu_1},\ldots,p^{\nu_k})$. Let
$g=\mu_k\ast_k f$, that is,
\begin{equation*}
g(n_1,\ldots,n_k) =  \sum_{d_1\mid n_1,\ldots, d_k\mid n_k} \mu(d_1)\cdots \mu(d_k) \left[n_1/d_1, \ldots, n_k/d_k \right].
\end{equation*}

Then $g$ is multiplicative and for any prime powers $p^{\nu_1},\ldots, p^{\nu_k}$ ($\nu_1,\ldots,\nu_k\ge 0$),
\begin{equation*}
g(p^{\nu_1},\ldots,p^{\nu_k}) =  \sum_{d_1,\ldots, d_k\in \{1,p\}} \mu(d_1)\cdots \mu(d_k) \left[p^{\nu_1}/d_1, \ldots, p^{\nu_k}/d_k \right].
\end{equation*}

Assume that there is $j\ge 1$ such that $\nu_1= \nu_2=  \cdots = \nu_j=\nu > \nu_{j+1}\ge \nu_{j+2}\ge \cdots \ge \nu_m\ge 1$,
$\nu_{m+1}=\cdots =\nu_k=0$. Then we have for any $d_1,\ldots,d_m\in \{1,p\}$, $d_{m+1},\ldots,d_k=1$,
\begin{equation*}
[p^{\nu_1}/d_1,\ldots,p^{\nu_k}/d_k]= \begin{cases} p^{\nu-1}, & \text{ if $d_1=\cdots =d_j=p$,}\\
p^\nu, & \text{ otherwise,}
\end{cases}
\end{equation*}
and
\begin{equation*}
g(p^{\nu_1},\ldots,p^{\nu_k}) = \Big(  p^\nu \sum_{d_1\in \{1,p\}} \mu(d_1) \cdots \sum_{d_j\in \{1,p\}} \mu(d_j) - p^\nu \mu(p)^j +p^{\nu-1} \mu(p)^j \Big)
\end{equation*}
\begin{equation*}
\times \sum_{d_{j+1}\in \{1,p\}} \mu(d_{j+1}) \cdots \sum_{d_m\in \{1,p\}} \mu(d_m) =
\begin{cases} (-1)^{j-1}(p^\nu-p^{\nu-1}), & \text{ if $j=m$,}\\
0, & \text{ otherwise.}
\end{cases}
\end{equation*}

Therefore, since $g$ is symmetric in the variables, we deduce
\begin{equation} \label{val_g}
g(p^{\nu_1},\ldots,p^{\nu_k})=\begin{cases} 1, & \text{ if $\nu_1=\cdots = \nu_k=0$,}\\
(-1)^{j-1} \varphi(p^\nu), & \text{ if a number $j\ge 1$ of $\nu_1,\ldots,\nu_k$ is equal to $\nu \ge 1$, }\\
& \text{ while all others are zero,} \\
0, & \text{ otherwise.}
\end{cases}
\end{equation}

Furthermore, let
\begin{equation*}
h(n_1,\ldots,n_k)= \sum_{d_1\mid n_1, \ldots, d_k\mid n_k} \frac{g(d_1,\ldots,d_k)}{ \varphi(d_1)\cdots \varphi(d_k)},
\end{equation*}
which is also multiplicative and symmetric in the variables. Let $p^{\nu_1}, \ldots, p_k^{\nu_k}$ be any prime powers and
assume, without loss of generality, that for some $t\ge 0$, one has $\nu_1\ge \cdots \ge \nu_t\ge 1$, $\nu_{t+1}=\cdots =
\nu_k=0$.

If $t=0$, then $h(1,\ldots,1)=1$. If $t\ge 1$, then
\begin{equation*}
h(p^{\nu_1},\ldots,p^{\nu_k})= \sum_{d_1\mid p^{\nu_1}, \ldots, d_t\mid p^{\nu_t}}
\frac{g(d_1,\ldots,d_t,1,\ldots,1)}{ \varphi(d_1)\cdots \varphi(d_t)}.
\end{equation*}

Let $d_1=p^{\beta_1},\ldots, d_t=p^{\beta_t}$, with $0\le \beta_1\le \nu_1,\ldots, 0\le \beta_t\le \nu_t$. For any subset $M$ of $\{1,\ldots,t\}$ such that $\#M=j$
($1\le j\le t$) let $\beta_m=\nu$ ($1\le \nu \le \nu_{\max M}$) for every $m\in M$ and $\beta_m=0$ for $m\notin M$. Then, according to \eqref{val_g},
\begin{equation*}
\frac{g(d_1,\ldots,d_t,1,\ldots,1)}{ \varphi(d_1)\cdots \varphi(d_t)} =\frac{(-1)^{j-1}\varphi(p^\nu)}{\varphi(p^\nu)^j}=
\frac{(-1)^{j-1}}{\varphi(p^\nu)^{j-1}}.
\end{equation*}

We deduce that
\begin{equation*}
h(p^{\nu_1},\ldots,p^{\nu_k})= 1+\sum_{j=1}^t \sum_{\substack{M\subseteq \{1,\ldots,t\}\\ \#M=j}}
\sum_{\nu=1}^{\nu_{\max M}} \frac{(-1)^{j-1}}{\varphi(p^\nu)^{j-1}}
\end{equation*}
\begin{equation*}
= 1+\sum_{j=1}^t (-1)^{j-1} \sum_{\substack{M\subseteq \{1,\ldots,t\}\\ \#M=j}}
\sum_{\nu=1}^{\nu_{\max M}} \frac1{\varphi(p^\nu)^{j-1}}.
\end{equation*}

Here, with the notation $A:=\nu_{\max M}$, we have for $j\ge 2$,
\begin{equation*}
K_j:= \sum_{\nu=1}^{\nu_{\max M}} \frac1{\varphi(p^\nu)^{j-1}}= \frac1{(p-1)^{j-1}} \sum_{\nu=1}^A \frac1{p^{(j-1)(\nu-1)}}
\end{equation*}
\begin{equation*}
= \frac{p^{j-1}}{(p-1)^{j-1}(p^{j-1}-1)}\left(1-\frac1{p^{A(j-1)}} \right),
\end{equation*}
and for $j=1$, $K_1=A$.

That is,
\begin{equation*}
h(p^{\nu_1},\ldots,p^{\nu_k})= 1+(\nu_1+\cdots +\nu_t)+ \sum_{j=2}^t \frac{(-1)^{j-1}p^{j-1}} {(p-1)^{j-1}(p^{j-1}-1)}
\sum_{\substack{M\subseteq \{1,\ldots,t\}\\ \#M=j}} \left(1-\frac1{p^{A(j-1)}} \right)
\end{equation*}
\begin{equation*}
= 1+(\nu_1+\cdots +\nu_t)+ \sum_{j=1}^{t-1} \frac{(-1)^jp^j}{(p-1)^j(p^j-1)} \Big( \binom{t}{j+1} -
\sum_{\substack{M\subseteq \{1,\ldots,t\}\\ \#M=j+1}} \frac1{p^{Aj}}\Big).
\end{equation*}
\end{proof}

\begin{cor} \label{Cor_lcm_n} {\rm ($n_1=\cdots =n_k=n$)}
For any $n,k\in \N$,
\begin{equation*}
\sum_{\substack{1\le a_1\le n\\ (a_1,n)=1}} \cdots \sum_{\substack{1\le a_k\le n\\ (a_k,n)=1}}
[(a_1-1,n), \ldots,(a_k-1,n)]
\end{equation*}
\begin{equation*}
=  \varphi(n)^k \prod_{p^\nu \mid \mid  n }
\left( 1+k\nu + \sum_{j=1}^{k-1} (-1)^j \binom{k}{j+1} \frac{p^j}{(p-1)^j(p^j-1)} \left(1-\frac1{p^{\nu j}}\right) \right).
\end{equation*}
\end{cor}

In the case $k=2$ this gives the formula \eqref{a_b_lcm}, while for $k=1$ we reobtain Menon's identity \eqref{Menon_id}.

\section{Acknowledgment} The second author was supported by the European Union, co-financed by the European
Social Fund EFOP-3.6.1.-16-2016-00004.

\end{document}